\documentclass[11pt]{amsart}
\usepackage{amsmath,amsfonts,amsbsy,amsgen,amscd,mathrsfs,amssymb,amsthm}
\usepackage{enumerate,mathtools}
\usepackage{cases}
\usepackage{dsfont}

\usepackage{clipboard}
\newclipboard{myclipboard}

\usepackage[T1]{fontenc} 

\usepackage{fullpage}
\usepackage{url}

\usepackage{mathrsfs}

\usepackage[colorlinks=true,linkcolor=blue, urlcolor=black]{hyperref} 
\makeatletter
\renewcommand*{\eqref}[1]{%
  \hyperref[{#1}]{\textup{\tagform@{\ref*{#1}}}}%
}
\makeatother

\usepackage{tikz}
\usetikzlibrary{calc}
\usetikzlibrary{intersections}
\usetikzlibrary{positioning}
\usetikzlibrary{hobby}

\usetikzlibrary{shadings}

\usetikzlibrary{decorations}
\numberwithin{figure}{section}

\numberwithin{equation}{section}

\newtheorem{theorem}{Theorem}[section]
\newtheorem*{theorem*}{Theorem}
\newtheorem{proposition}[theorem]{Proposition}

\newtheorem{corollary}[theorem]{Corollary}
\newtheorem*{corollary*}{Corollary}
\newtheorem{lemma}[theorem]{Lemma}
\theoremstyle{definition}
\newtheorem{definition}[theorem]{Definition}

\newtheorem{remark}[theorem]{Remark}

\newtheorem*{claim*}{Claim}

\newcommand{\R}{\mathbb R}
\newcommand{\N}{\mathbb N}
\newcommand{\Z}{\mathbb Z}

\newcommand{\EE}{\mathbb E}

\newcommand{\leb}{\operatorname{Leb}}
\newcommand{\Ent}{\operatorname{Ent}}
\newcommand{\prob}{\mu}
\newcommand{\den}{f}
\newcommand{\Den}{F}
\newcommand{\state}{\mathbb X}
\newcommand{\sig}{\mathcal X}
\newcommand{\pos}{\pi}
\newcommand{\ins}{\lambda}
\newcommand{\TT}{T}
\newcommand*\diff{\mathop{}\!\mathrm{d}}
\newcommand{\Pheat}{\mathrm P}
\newcommand{\PheatE}{\mathrm H}
\newcommand{\derm}{\mathrm D}
\newcommand{\der}{D}
\newcommand{\kk}{k}
\newcommand{\ubd}{M}
\newcommand{\pstate}{\Omega}
\newcommand{\psig}{\mathcal F}
\newcommand{\pmeasure}{\mathbb P}
\newcommand{\XX}{X}
\newcommand{\YY}{Y}
\newcommand{\LL}{|\log\prob(0)|}

\title{The Poisson transport map}

\author{Pablo L\'{o}pez-Rivera}

\address{Laboratoire Jacques-Louis Lions (LJLL),  Universit\'{e} Paris Cit\'{e} \& Sorbonne Universit\'{e}, CNRS, Paris
F-75013, France}
\email{plopez@math.univ-paris-diderot.fr}

\author{Yair Shenfeld}
\address{Division of Applied Mathematics, Brown University, Providence, RI, USA}
\email{Yair\_Shenfeld@Brown.edu}

\begin{document}
\maketitle

\begin{abstract}
We construct a transport map from Poisson point processes onto ultra-log-concave measures over the natural numbers, and show that this map is a contraction.  Our approach overcomes the known obstacles to transferring functional inequalities using transport maps in discrete settings, and allows us to deduce a number of functional inequalities for ultra-log-concave measures. 
In particular, we provide the currently best known constant in modified logarithmic Sobolev inequalities for ultra-log-concave measures.
\end{abstract}

\section{Introduction}
\label{sec:intro}

\subsection{The Poisson transport map}
A classical way to establish functional inequalities for a given probability measure is to find a \emph{Lipschitz} transport map from a source measure, for which the inequality is known, onto the target measure of interest. For example, suppose we want to prove a logarithmic Sobolev inequality for a probability measure $\prob$ over $\R^d$. Suppose further that we can find an $L$-Lipschitz map $\XX:\R^n\to \R^d$, where $n\ge d$, which transports the standard Gaussian $\gamma_n$ on $\R^n$ to $\prob$. Since $\gamma_n$ satisfies a logarithmic Sobolev inequality with constant $2$, we can write, for $g:\R^d\to \R$,
\begin{align}
\label{eq:LSI_trnasport}
\Ent_{\prob}(g^2)=\Ent_{\gamma_n}(g^2\circ \XX)\le 2\int_{\R^n}|\nabla(g\circ \XX)|^2\diff \gamma_n\le 2L^2\int_{\R^n}|(\nabla g)\circ \XX)|^2\diff \gamma_n=2L^2\int_{\R^d}|\nabla g|^2\diff\prob.
\end{align}
Thus, $\prob$ satisfies a logarithmic Sobolev inequality with constant $2L^2$. If we wish to apply this method for discrete measures then we face a number of obstacles. Consider for instance the problem of constructing a Lipschitz transport map $\XX:\N\to\N$ between the Poisson measure (with intensity 1) $\pos_1$ on $\N$, and another probability measure $\prob$ on $\N$. The fact that the map $\XX$ cannot split the mass of $\pos_1$ at any position in $\N$ severely restricts the type of measures $\prob$ that can arise as the pushforward of $\pos_1$ under $\XX$. In addition, even if we can construct a Lipschitz transport map $\XX$ between $\pos_1$ and $\prob$, the lack of chain rule in the discrete setting  hinders the argument in \eqref{eq:LSI_trnasport}. 

In this work we show that these obstacles can be overcome by transporting the Poisson \emph{point processes} $\pmeasure$  onto probability measures $\prob$ on $\N$. In the notation above, $d=1$ and  $n= \infty$. In addition, we will show that in the setting considered in this work, the chain rule issue can be avoided. Let us describe informally our transport map. Fix a time $\TT>0$ and $\ubd>0$, and consider a Poisson point process over $[0,\TT]\times [0,\ubd]$: 
\begin{itemize}
\item The numbers of points that fall in disjoint regions of  $[0,\TT]\times [0,\ubd]$ are independent.
\item Given $B\subseteq [0,\TT]\times [0,\ubd]$, the number of points that fall into $B$ is distributed like a Poisson measure on $\N$ with intensity $\leb(B)$.
\end{itemize}
Now let $\ins:[0,\TT]\to [0,\ubd]$ be a regular curve, and define the counting process $(\XX_t^{\ins})_{t\in [0,\TT]}$ by letting 
\begin{equation}
\label{eq:X_t_def_intro}
X_t^{\ins}:=\text{number of points in $[0,t]\times [0,\ubd]$ that fall below the curve $\ins$},\quad \text{(Figure \ref{fig:X})}.
\end{equation}
Given a measure $\prob$ on $\N$, we can choose $\ins$ in a stochastic way so that $\XX_{\TT}^{\ins}\sim \prob$. We call $\XX_{\TT}^{\ins}$ the \textbf{Poisson transport map} as it transports the Poisson point process $\pmeasure$ onto $\prob$. 

The Poisson transport map can be viewed as the discrete analog of the \emph{Brownian transport map} of Mikulincer and the second author \cite{mikulincer2021brownian}, which transports the Wiener measure on path space onto  probability measures over $\R^d$. The Brownian transport map  is based on the F\"ollmer process \cite{follmer2005entropy, follmer2006time, lehec2013representation}, and, analogously, the Poisson transport map is based on the process $(\XX_t^{\ins})_{t\in [0,T]}$, which is the discrete analogue of the F\"ollmer process. The process $(\XX_t^{\ins})_{t\in [0,T]}$ was constructed by Klartag and Lehec \cite{Klartag_Lehec} (specializing and elaborating on earlier work of Budhiraja, Dupuis, and Maroulas \cite{MR2841073}), who  used it to prove functional inequalities. In Section \ref{subsec:BrownianVsPoisson}, we discuss the similarities and differences between the Brownian transport map and the Poisson transport map. 

\begin{figure}
\begin{tikzpicture}[scale=1]
\draw[->] (0,0) -- (5,0);
  \draw[->] (0,0) -- (0,5);
  \draw[dashed] (0,4.8) -- (5,4.8);
  \draw (1.4,0.2) -- (1.4,-0.2);
\draw [black] plot [smooth, tension=1] coordinates { (0,.5) (1,4) (3.5,2) (4.2,4.5) (5,3.8)};
\node at (-0.4,4.8) {$\ubd$};
\node at (5,-0.3) {$\TT$};
\node at (1.4,-0.4) {$2$};
\node at (6,3.8) {$(\ins_t)_{t\in [0,\TT]}$};
\draw[black,] (0.2,4) circle (.3ex);
\draw[black,fill=black] (1,1) circle (.3ex);
\draw[black,fill=black] (2,2) circle (.3ex);
\draw[black,] (2.8,3) circle (.3ex);
\draw[black,fill=black] (3.3,1) circle (.3ex);
\draw[black,] (3.8,3.3) circle (.3ex);
\draw[black,fill=black] (4.5,4) circle (.3ex);
\end{tikzpicture}
\caption{The points in $[0,\TT]\times [0,\ubd]$ are generated according to a standard Poisson process (7 points in this case). At time $t\in [0,\TT]$ the value of the process $\XX_t^{\ins}$ is equal to the number of points under the curve (filled circles). In the figure, $X_2^{\ins}=1$ and $\XX_{\TT}^{\ins}=4$.}
\label{fig:X}
\end{figure}
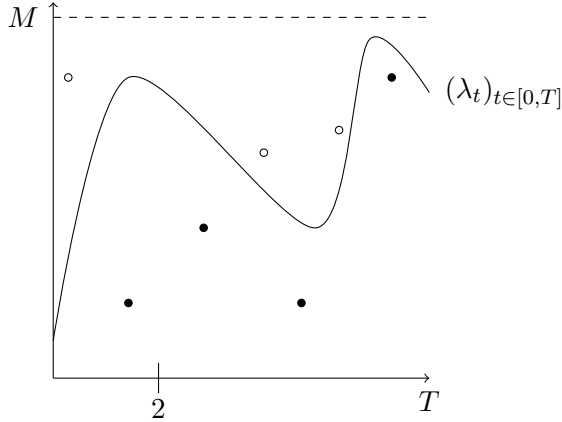

\subsection{Ultra-log-concave measures}

Just as in the continuous case, we cannot expect to have a Lipschitz transport  map (with good constants) from $\pmeasure$ onto any probability measure $\prob$ on $\N$, since the existence of such map will imply functional inequalities for $\prob$. The classical result on the existence of Lipschitz transport maps in the continuous setting is due to Caffarelli \cite{MR1800860}, who showed that if $n=d$, and $\prob=\den\gamma_d$ with $f:\R^d\to\R_{\ge 0}$ log-concave, then there exits a 1-Lipschitz transport map between $\gamma_d$ and $\prob$. Closer to our setting, it was shown in \cite{mikulincer2021brownian} that the Brownian transport map is  1-Lipschitz when the target measure over $\R^d$ is of the form $\prob=\den\gamma_d$, with $f$ log-concave.

In the discrete setting, the analogue of a measure $\prob$ being ``more log-concave than the Gaussian" is that the measure is \emph{ultra-log-concave}. To define this notion we recall that a positive function $\den:\N\to \R_{> 0}$ is \emph{log-concave} if
\begin{equation}
\label{eq:lc_def_ulc_intro}
\den^2(\kk)\ge \den(\kk-1)\den(\kk+1),\quad \forall ~\kk\in\{1,2,\ldots\}.
\end{equation}
\begin{definition}
\label{def:ULC_intro}
A probability measure $\prob$ on $\N$ is \emph{ultra-log-concave} if $\prob=\den\pos_{\ins}$, where $\pos_{\ins}$ is the Poisson measure with intensity $\ins$, and $\den:\N\to \R_{>0}$ is a positive  log-concave function\footnote{In Section \ref{sec:ulc} we recall equivalent definitions of ultra-log-concave measures.}.
\end{definition}
Ultra-log-concave measures  form an important class of discrete probability measures as it possesses desirable properties such as closure under convolution \cite[Theorem 4.1(b)]{MR3290441}. They are also ubiquitous and show up in fields outside of probability such as  combinatorics and convex geometry. We refer to the introduction of \cite{MR4420904} for more information. 

Our first main result is that the Poisson transport map from the Poisson point process $\pmeasure$ onto ultra-log-concave measures $\prob$ is 1-Lipschitz. We will formulate this condition in terms of the Malliavin derivative $\derm_{(t,z)}$ of $\XX_{\TT}$, which captures the effect of adding a point at $(t,z)$ to the point process on the value of $\XX_{\TT}$ (see Section \ref{subsec:poisson_space}).

\begin{theorem}
\label{thm:contraction_intro}
Fix a real number $\TT>0$, and let $\prob=\den\pos_{\TT}$ be an ultra-log-concave probability measure over $\N$. Let $\XX_{\TT}$ be the Poisson transport map from $\pmeasure$ to $\prob$. Then, $\pmeasure$-almost-surely,
\[
\derm_{(t,z)}\XX_{\TT}\in \{0,1\}\quad\forall~ (t,z)\in [0,\TT]\times [0,\ubd],\quad \left(\ubd=\frac{\den(1)}{\den(0)}\right).
\]
\end{theorem}
The fact that $\derm_{(t,z)}\XX_{\TT}$ is integer-valued follows from the definition of the Malliavin derivative $\derm_{(t,z)}$, and since $\XX_{\TT}$ is integer-valued. However, a priori, saying that $\XX_{\TT}$ is 1-Lipschitz could have implied $\derm_{(t,z)}\XX_{\TT}\in \{-1,0,1\}$. Theorem \ref{thm:contraction_intro} shows that $\derm_{(t,z)}\XX_{\TT}\ge 0$, which will be important to tackle the chain rule issue when transporting functional inequalities from $\pmeasure$ to $\prob$.

\subsection{Functional inequalities for ultra-log-concave measures} The absence of the chain rule in the discrete setting  complicates the study of functional inequalities for  measures on $\N$.  For example, Poisson measures $\pos_{\ins}$ over $\N$, the discrete analogues of Gaussians, do not satisfy logarithmic Sobolev inequalities. Rather, they satisfy \emph{modified} logarithmic Sobolev inequalities as was first developed by Bobkov and Ledoux \cite{MR1636948}. In the discrete setting there are various choices for modified logarithmic Sobolev inequalities (see \cite[Proposition 3.6]{bobkov2006modified} for the relation between them), and in the context of the Poisson measure \emph{Wu's inequality} is  the strongest. For example, to recover the Gaussian  logarithmic Sobolev inequality via the combination of the Poisson  modified logarithmic Sobolev inequality and the central limit theorem one needs Wu's inequality \cite[Remark on page 75]{LSIessentials}. To introduce Wu's inequality let $\der $ be the discrete derivative of a function $g:\N\to \R$,
\[
\der g(\kk):=g(\kk+1)-g(\kk)\quad\text{for} \quad \kk\in\N.
\]
\begin{theorem}
\label{thm:Wu_intro}{
 \cite[Theorem 1.1]{MR1800540}.}
Let $\pos_{\TT}$ be the Poisson measure over $\N$ with intensity $\TT$. Then, for any positive $g\in L^2(\N,\pos_{\TT})$,
\begin{equation}
\label{eq:Wu_intro}
\Ent_{\pos_{\TT}}(g)\le \TT\,\EE_{\pos_{\TT}}[\Psi(g,\der g)],
\end{equation}
where $\Psi(u,v):=(u+v)\log(u+v)-u\log u-(\log u+1)v$.
\end{theorem}
In the continuous setting, as a consequence  of the existence of 1-Lipschitz transport maps, measures which are more log-concave than Gaussians satisfy logarithmic Sobolev inequalities. Thus, in the discrete setting we can expect ultra-log-concave measures to satisfy modified logarithmic Sobolev inequalities. Indeed, such inequalities for ultra-log-concave measures were obtained by Caputo, Dai Pra, and Posta \cite[Theorem 3.1]{caputo2009convex}, but the stronger Wu-type modified logarithmic Sobolev inequality for ultra-log-concave measures was only obtained later by Johnson in \cite{MR3729642}. 

\begin{theorem}{\cite[Theorem 1.3 and Lemma 5.1]{MR3729642}.}
\label{thm:Johnson_into}
Let $\prob$ be an ultra-log-concave probability measure over $\N$. Then, for any positive $g\in L^2(\N,\prob)$,
\begin{equation}
\label{eq:Johnson _intro}
\Ent_{\prob}(g)\le \frac{\prob(1)}{\prob(0)}\,\EE_{\prob}[\Psi(g,\der g)],
\end{equation}
where $\Psi(u,v):=(u+v)\log(u+v)-u\log u-(\log u+1)v$.
\end{theorem}
Note that  $\frac{\prob(1)}{\prob(0)}=\TT$ when $\prob=\pos_{\TT}$, so \eqref{eq:Wu_intro} and \eqref{eq:Johnson _intro} agree in this case. Our second main result shows that we can in fact improve the constant in the strongest modified logarithmic Sobolev inequalities for ultra-log-concave measures. 
\begin{theorem}
\label{thm:modified_LSI_ULC_into}
Let $\prob$ be an ultra-log-concave probability measure over $\N$. Then, for any positive $g\in L^2(\N,\prob)$,
\begin{equation}
\label{eq:modified _intro}
\Ent_{\prob}(g)\le \LL\,\EE_{\prob}[\Psi(g,\der g)],
\end{equation}
where $\Psi(u,v):=(u+v)\log(u+v)-u\log u-(\log u+1)v$.
\end{theorem}
It will follow from our work (Corollary \ref{cor:mean_mu}) that $\LL\le  \frac{\prob(1)}{\prob(0)}$, so that \eqref{eq:modified _intro} improves on \eqref{eq:Johnson _intro}. (Note however that \eqref{eq:Johnson _intro} holds, with constant $1/c$, for the larger class of \emph{$c$-log-concave measures} \cite{MR3729642}.) Again, when $\prob=\pos_{\TT}$,  we have $\LL=\TT$.

\begin{remark}[The optimal constant]
\label{rem:opt}
Theorem \ref{thm:modified_LSI_ULC_into} raises the question of what is the optimal constant in modified logarithmic Sobolev inequalities for ultra-log-concave measures.  Fraser and Johnson  \cite[Corollary 2.4]{MR3006983} showed that the Poincar\'e inequality for ultra-log-concave measures  holds with a constant at least as good as $\EE[\prob]:=\EE_{Z\sim \prob}[Z]$. On the other hand, it will follow from our work (Corollary \ref{cor:mean_mu}) that 
\[
\EE[\prob]\le \LL\le \frac{\prob(1)}{\prob(0)},
\]
which begs the question of whether \eqref{eq:modified _intro} holds with constant $\EE[\prob]$. As evidence for an affirmative answer, it was shown by Aravinda, Marsiglietti, and Melbourne \cite[Theorem 1.1]{MR4420904}  that ultra-log-concave measures satisfy concentration inequalities with Poisson tail bounds. On the other hand, if 
the modified logarithmic Sobolev inequalities were to hold for ultra-log-concave measures with constant $\EE[\prob]$, the result  \cite[Theorem 1.1]{MR4420904} could be deduced from  the usual Herbst argument. 
\end{remark}

Theorem \ref{thm:modified_LSI_ULC_into} is in fact a corollary of the following more general result, namely, the validity of  Chafa\"{\i}'s  $\Phi$-Sobolev inequalities for ultra-log-concave measures; see Section \ref{subsec:Phi_Sobolev} for the precise definitions. 

\begin{theorem}
\label{thm:PhiSoboleUlc_intro}
Let $\prob$ be an ultra-log-concave probability measure over $\N$. Let $\mathcal I\subseteq \R$ be a closed interval, not necessarily bounded, and let $\Phi:\mathcal I\to \R$ be a smooth convex function. Suppose that the function
\[
\{(u,v)\in \R^2:(u,u+v)\in \mathcal I\times\mathcal I\}\ni (u,v)\quad\mapsto \quad \Psi(u,v):=\Phi(u+v)-\Phi(u)-\Phi'(u)v
\]
is nonnegative and convex. Then, for any $g\in L^2(\N,\prob)$, such that $\prob$-a.s. $g,g+\der g\in\mathcal I$,
\begin{equation}
\label{eq:PhiSobolevUlc_intro}
\Ent^{\Phi}_{\prob}(g)\le \LL\,\EE_{\prob}[\Psi(g,\der g)].
\end{equation}
\end{theorem}

We conclude with the following transport-entropy inequality for ultra-log-concave measures; see Section \ref{subsec:entropy-transport} for the precise definitions. 

\begin{theorem}
\label{thm:entropy-transport_ULC_intro}
Let $\prob=\den\pos_{\TT}$ be an ultra-log-concave probability measure on $\N$, and let  $\ubd:=\frac{\den(1)}{\den(0)}$. Then,  for any probability measure $\nu$ on $\N$ which is absolutely continuous with respect to $\prob$, and has a finite first moment, we have 
\begin{equation}
\label{eq:transport_inq_ULC_intro}
\alpha_{\TT\ubd}\left(W_{1,|\cdot|}(\nu,\prob)\right)\le H(\nu|\prob),
\end{equation}
where 
\[
\alpha_c(r):=c\left[\left(1+\frac{r}{c}\right)\log\left(1+\frac{r}{c}\right)-\frac{r}{c}\right].
\]
\end{theorem}
The constant $\TT\ubd$ in \eqref{eq:transport_inq_ULC_intro} can in fact be improved; cf. Remark \ref{rem:better_c}.

\subsection*{Organization of paper} In Section \ref{sec:ulc}
 we review some of the basics of ultra-log-concave measures, as well as the basics of the  Poisson semigroup. Section \ref{sec:poisson_transport} provides the construction of the Poisson transport map, as well as some of its properties. In Section \ref{sec:contraction} we prove our  contraction theorem (Theorem \ref{thm:contraction_intro}). In addition, in Section \ref{subsec:BrownianVsPoisson}, we compare and contrast the Brownian transport map and the Poisson transport map. Finally, in Section \ref{sec:func_inq} we prove our functional inequalities (Theorem \ref{thm:modified_LSI_ULC_into}, Theorem \ref{thm:PhiSoboleUlc_intro}, and Theorem \ref{thm:entropy-transport_ULC_intro}). 

\subsection*{Acknowledgments} 
We are grateful to Joseph Lehec, Arnaud Marsiglietti, and Avelio Sep\'{u}lveda for their valuable comments. We would like to extend special thanks to Nachi Avraham Re'em and Max Fathi for their many helpful remarks on this manuscript. We are also very grateful to the referees whose comments have significantly improved the quality of this work. 

This project has received funding from the European Union's Horizon 2020 research and innovation programme under the Marie Sk\l{}odowska-Curie grant agreement No 945332. This work has also received support under the program ``Investissement d'Avenir" launched by the French Government and implemented by ANR, with the reference ``ANR-18-IdEx-0001" as part of its program ``Emergence".  This work received funding from the Agence Nationale de la Recherche (ANR) Grant ANR-23-CE40-0003 (Project CONVIVIALITY), as well as funding from the Institut Universitaire de France.

This material is based upon work supported by the National Science Foundation under awards DMS-2331920 and 2002022. 
\section{Ultra-log-concave measures}
\label{sec:ulc}

In this section we establish some of the properties of ultra-log-concave measures that will be used throughout the paper. We will denote $\N:=\{0,1,2,\ldots\}$ the set of nonnegative integers, $\Z_+:=\{1,2,\ldots\}$, and by  $\pos_{\ins}$ the Poisson measure on $\N$ with intensity $\ins>0$,
\[
\pos_{\ins}(\kk):=e^{-\ins}\frac{\ins^{\kk}}{\kk!},\quad \kk\in \N.
\]
We say that a positive function $\den:\N\to \R_{> 0}$ is \emph{log-concave} if 
\begin{equation}
\label{eq:lc_def_ulc}
\den^2(\kk)\ge \den(\kk-1)\den(\kk+1),\quad \forall ~\kk\in\Z_+.
\end{equation}
Equivalently, $\den:\N\to \R_{> 0}$ is log-concave if the function
\begin{equation}
\label{eq:lc_def_ulc_mono}
\Z_+\ni\kk\mapsto \frac{\den(\kk)}{\den(\kk-1)}\text{ is non-increasing}.
\end{equation}

The following definition captures the intuition of a probability measure being more log-concave than a Poisson measure. 
\begin{definition}
\label{def:ulc}
A probability $\prob$ on $\N$ is \emph{ultra-log-concave} if there exists $\ins>0$, and a  positive  log-concave function $\den$, such that $\prob(\kk)=\den(\kk)\pos_{\ins}(\kk)$ for all $\kk\in \N$.
\end{definition}
The intensity $\ins$ in Definition \ref{def:ulc} does not in fact play any role. It is readily verified from the definition that $\prob$ is ultra-log-concave, with respect to \emph{any} intensity $\ins>0$, if and only if,
\begin{equation}
\label{eq:ulc}
\prob^2(\kk)\ge\frac{\kk+1}{\kk}\prob(\kk+1)\prob(\kk-1),\quad \forall ~\kk\in\Z_+.
\end{equation}
In other words, once $\prob$ is more log-concave than $\pos_{\ins}$ \emph{for some} $\ins$, it is in fact more log-concave than $\pos_{\ins}$ \emph{for all} $\ins$.

The \emph{Poisson semigroup} $(\Pheat_t)_{t\ge 0}$ will play an important role in our work. Given a function $g:\N\to\R$ we define, for $t\ge 0$,
\[
\Pheat_0g:=g,\quad\text{and}\quad \Pheat_tg(\kk):=\sum_{n=0}^{\infty}g(\kk+n)\pos_t(n),\quad \forall ~ \kk\in \N,\quad t>0.
\]
The Poisson semigroup satisfies the identity
\begin{equation}
\label{eq:poisson_pde}
\partial_t(\Pheat_tg)(\kk)=\der(\Pheat_tg)(\kk),\quad\forall ~\kk\in \N,
\end{equation}
where
\[
\der h(\kk):=h(\kk+1)-h(\kk),
\]
for any $h:\N\to\R$. Fix a time $T>0$. For future reference, given nonnegative $\den:\N\to \R$, we set 
\begin{equation}
\label{eq:F_def}
\Den(t,\kk):=\log \Pheat_{\TT-t}\den(\kk)\quad\text{which satisfies}\quad \partial_t\Den(t,\kk)=-e^{\der \Den(t,\kk)}+1,\quad\forall ~t\in [0,\TT], ~\kk\in \N.
\end{equation}

Our next result shows that the Poisson semigroup preserves log-concavity. While a number of proofs are available, our proof will mimic the proof of the fact that the heat semigroup preserves log-concavity. The latter is a consequence of the Pr\'ekopa-Leindler  inequality, so we will use a discrete analogue of the Pr\'ekopa-Leindler  inequality proven by Klartag and Lehec. 
\begin{proposition}
\label{prop:poisson_presev}
$\den:\N\to \R_{>0}$ be a log-concave function. Then, for any $t\ge 0$, $\Pheat_t\den$ is a log-concave function.
\end{proposition}
\begin{proof}
Let $V:=\log \den$. Our goal is to show that
\[
(\Pheat_t e^V(\kk))^2\ge \Pheat_t e^V(\kk+1)\Pheat_t e^V(\kk-1)\quad\forall~\kk\in \Z_+,
\]
which, by definition, is equivalent to 
\begin{equation}
\label{eq:heat_lc_proof}
\left(\sum_{n=0}^{\infty}e^{V(\kk+n)}\pos_t(n)\right)^2\ge \left(\sum_{n=0}^{\infty}e^{V(\kk+1+n)}\pos_t(n)\right)\left(\sum_{n=0}^{\infty}e^{V(\kk-1+n)}\pos_t(n)\right).
\end{equation}
The discrete Pr\'ekopa-Leindler  inequality \cite[Proposition 5.1]{Klartag_Lehec} implies that for all functions $W,Y,Z:\N\to \R$,
\begin{equation}
\begin{split}
&W(\ell)+Y(m)\le Z\left(\left\lfloor\frac{\ell+m}{2}\right\rfloor\right)+Z\left(\left\lceil\frac{\ell+m}{2}\right\rceil\right),\quad \forall~\ell,m\in \N,\label{eq:DPL}\\
&\Longrightarrow \\
& \left(\sum_{n=0}^{\infty}e^{Z(n)}\pos_t(n)\right)^2\ge \left(\sum_{n=0}^{\infty}e^{W(n)}\pos_t(n)\right)\left(\sum_{n=0}^{\infty}e^{Y(n)}\pos_t(n)\right).
\end{split}
\end{equation}
To apply \eqref{eq:DPL} we fix $\kk\in \N$ and define $Z(n):=V(k+n), ~ W(n):=V(k+1+n)$, and $Y(n):=V(k-1+n)$, so that to establish \eqref{eq:heat_lc_proof} it suffices to show 
\begin{equation}
\label{eq:DPL_ass}
V(k+1+\ell)+V(k-1+m)\le V\left(k+\left\lfloor\frac{\ell+m}{2}\right\rfloor\right)+V\left(k+\left\lceil\frac{\ell+m}{2}\right\rceil\right)\quad\forall~\ell,m\in \N.
\end{equation}
To verify \eqref{eq:DPL_ass}  we note that the log-concavity of $\den$ implies that 
\begin{equation}
\label{eq:concave_def1}
2V(m)\ge V(m+1)+V(m-1)\quad\forall~m\in \Z_+,
\end{equation}
which is in fact  equivalent to 
\begin{equation}
\label{eq:concave_def2}
V(p)+V(q)\le V\left(\left\lfloor\frac{p+q}{2}\right\rfloor\right)+V\left(\left\lceil\frac{p+q}{2}\right\rceil\right)\quad\forall~p,q\in \N.
\end{equation}
Taking $p=k+1+\ell$ and $q=k-1+m$ in \eqref{eq:concave_def2} yields
\begin{align*}
V(k+1+\ell)+V(k-1+m)&\le V\left(\left\lfloor\frac{2k+\ell+m}{2}\right\rfloor\right)+V\left(\left\lceil\frac{2k+\ell+m}{2}\right\rceil\right)\\
&=V\left(k+\left\lfloor\frac{\ell+m}{2}\right\rfloor\right)+V\left(k+\left\lceil\frac{\ell+m}{2}\right\rceil\right),
\end{align*}
thus deducing \eqref{eq:DPL_ass}.
\end{proof}

As a consequence of the preservation of log-concavity let us deduce a number of corollaries which will be helpful later on. 
\begin{corollary}
\label{cor:mono_ratio}
Let $\den:\N\to \R_{> 0}$ be a log-concave function. Fix $\TT>0$ and $\kk\in \N$. The map
\[
[0,\TT]\ni t\mapsto \frac{\Pheat_{\TT-t}\den(\kk+1)}{\Pheat_{\TT-t}\den(\kk)}
\]
is non-decreasing.
\end{corollary}
\begin{proof}
Define $\theta: [0,\TT]\to \R$ by $\theta(t):=\frac{\Pheat_{\TT-t}\den(\kk+1)}{\Pheat_{\TT-t}\den(\kk)}$, so we need to show that $\theta'(t)\ge 0$. Indeed, by \eqref{eq:poisson_pde},
\begin{align*}
&\theta'(t)=\frac{-\partial_t(\Pheat_{\TT-t}\den)(\kk+1)}{\Pheat_{\TT-t}\den(\kk)}-\frac{\Pheat_{\TT-t}\den(\kk+1)(-\partial_t(\Pheat_{\TT-t}\den)(\kk))}{(\Pheat_{\TT-t}\den(\kk))^2}\\
&=\frac{\Pheat_{\TT-t}\den(\kk+1)(\der(\Pheat_{\TT-t}\den)(\kk))}{(\Pheat_{\TT-t}\den(\kk))^2}-\frac{\der(\Pheat_{\TT-t}\den)(\kk+1)}{\Pheat_{\TT-t}\den(\kk)}\\
&=\frac{1}{(\Pheat_{\TT-t}\den(\kk))^2}\left\{\Pheat_{\TT-t}\den(\kk+1)[\Pheat_{\TT-t}\den(\kk+1)-\Pheat_{\TT-t}\den(\kk)]-\Pheat_{\TT-t}\den(\kk)[\Pheat_{\TT-t}\den(\kk+2)-\Pheat_{\TT-t}\den(\kk+1)]\right\}\\
&=\frac{1}{(\Pheat_{\TT-t}\den(\kk))^2}\left\{(\Pheat_{\TT-t}\den)^2(\kk+1)-\Pheat_{\TT-t}\den(\kk)\Pheat_{\TT-t}\den(\kk+2)\right\}\ge 0,
\end{align*}
where the last inequality holds by Proposition \ref{prop:poisson_presev}. 
\end{proof}

\begin{corollary}
\label{cor:G_is_bounded}
Let $\den:\N\to \R_{> 0}$ be a log-concave function. Fix $\TT>0$. Then, for any $t\in [0,\TT]$ and $\kk\in\N$,
\begin{equation}
\label{eq:G_is_bounded}
 \frac{\Pheat_{\TT-t}\den(\kk+1)}{\Pheat_{\TT-t}\den(\kk)}\le  \frac{\den(1)}{\den(0)}.
\end{equation}
\end{corollary}

\begin{proof}
Fix $k\in \N$. By  Corollary \ref{cor:mono_ratio} the function $[0,\TT]\ni t\mapsto \frac{\Pheat_{\TT-t}\den(\kk+1)}{\Pheat_{\TT-t}\den(\kk)}$ is non-decreasing, so
\[
\forall~t\in [0,\TT],\quad \frac{\Pheat_{\TT-t}\den(\kk+1)}{\Pheat_{\TT-t}\den(\kk)}\le \frac{\den(\kk+1)}{\den(\kk)}.
\]
On the other hand, by \eqref{eq:lc_def_ulc_mono}, the function $\N\ni \kk\mapsto \frac{\den(\kk+1)}{\den(\kk)}$ is non-increasing, so 
\[
\forall~t\in [0,\TT],\quad \frac{\Pheat_{\TT-t}\den(\kk+1)}{\Pheat_{\TT-t}\den(\kk)}\le \frac{\den(\kk+1)}{\den(\kk)}\le  \frac{\den(1)}{\den(0)}.
\]
\end{proof}

\section{The Poisson transport map}
\label{sec:poisson_transport}
In this section we construct the Poisson transport map. In Section \ref{subsec:poisson_space} we recall the construction of the canonical space for the Poisson point process, as well as the basics of the Malliavin calculus on this space. We will use \cite{Last, Peccati_Bourguin} as our references. In Section \ref{subsec:poisson_transport} we describe the process $(\XX_t^{\ins})_{t\in [0,\TT]}$ constructed  by Klartag and Lehec \cite{Klartag_Lehec}, which we interpret as a transport map from the Poisson measure on the canonical space onto probability measures over $\N$. 

\subsection{The Poisson space}
\label{subsec:poisson_space}
Fix a real number $\TT>0$. Let $\prob=\den\pos_{\TT}$ be an ultra-log-concave probability measure on $\N$, where $\den:\N\to\R_{> 0}$ is a positive log-concave function. Set $\ubd:=\frac{\den(1)}{\den(0)}$, and let $\state:=[0,\TT]\times [0,\ubd]$.  We let $\sig$ be the sigma-algebra generated by the Borel sets of $\state$ endowed with the product topology, and we let $\leb$ be the Lebesgue measure on $\sig$. We define the Poisson space $(\pstate,\psig,\pmeasure)$ over $(\state,\sig,\leb)$ by letting the probability space be
\[
\pstate:=\left\{\omega:\omega=\sum_i\delta_{(t_i,z_i)},~ (t_i,z_i)\in \state ~(\text{at most countable})\right\},
\]
the sigma-algebra be
\[
\psig:=\sigma\left(\pstate\ni\omega\mapsto \omega(B): B\in \sig \right),
\]
and defining the probability measure $\pmeasure$ by 
\begin{align*}
&\forall~ B\in \sig,~ \forall~ \kk\in \N, \quad\pmeasure[\{\omega(B)=\kk\}]=\pos_{\leb(B)}(\kk),\\
&\forall ~n\in \Z_+,\quad \omega(B_1),\ldots,\omega(B_n)\text{ are $\pmeasure$-independent if $B_1,\ldots, B_n\in \sig$ are disjoint}.
\end{align*}
Given a measurable function $G:\pstate\to \R$, we define the Malliavin derivative $\derm$ of $G$ as the function $\derm G:\pstate\times \state\to \R$ given by
\begin{equation}
\label{eq:Malliavin_der_def}
\derm_{(t,z)} G(\omega):= G(\omega+\delta_{(t,z)})-G(\omega),\quad\forall~ (t,z)\in \state,\quad\forall~\omega\in \pstate. 
\end{equation}
Of particular importance to us will be binary Malliavin derivatives, for which one has the following chain rule.
\begin{lemma}
\label{lem:chain_rule}
Let $G:\pstate\to \N$ be a measurable function  such that $\derm_{(t,z)} G\in \{0,1\}$ for all $(t,z)\in \state$. Then, for any $g:\N\to \R$,
\begin{align}
\label{eq:chain_rule}
\derm_{(t,z)}(g\circ G)=\der g(G)\cdot\derm_{(t,z)}G\quad\forall~ (t,z)\in \state.
\end{align}
\end{lemma}
\begin{proof}
Fix $\omega\in\pstate$ and $(t,z)\in \state$. If $\derm_{(t,z)}G(\omega)=G(\omega+\delta_{(t,z)})-G(\omega)=0$, then
\[
\derm_{(t,z)}(g\circ G(\omega))=g(G(\omega+\delta_{(t,z)}))-g(G(\omega))=0,
\]
since $G(\omega+\delta_{(t,z)})=G(\omega)$, which establishes \eqref{eq:chain_rule}. Suppose then that $\derm_{(t,z)}G(\omega)=1$, so that $G(\omega+\delta_{(t,z)})=G(\omega)+1$. Then, 
\begin{align*}
\derm_{(t,z)}(g\circ G(\omega))&=g(G(\omega+\delta_{(t,z)}))-g(G(\omega))=g(G(\omega)+1)-g(G(\omega))=\der g(G(\omega))\\
&=\der g(G(\omega))\derm_{(t,z)} G(\omega),
\end{align*}
where in the last equality we used $\derm_{(t,z)} G(\omega)=1$. 
\end{proof}

\subsection{The  Poisson transport map}
\label{subsec:poisson_transport}
Our construction of the Poisson transport map is based on the stochastic process used by Klartag and Lehec in \cite{Klartag_Lehec} (whose origin can be found in Budhiraja, Dupuis, and Maroulas \cite{MR2841073}). Let the notation and assumptions of Section \ref{subsec:poisson_space} hold. Given $t\in [0,\TT]$ let $\sig_t$ be the sigma-algebra generated by the Borel sets of $[0,t]\times  [0,\ubd]$ endowed with the product topology, and define the sigma-algebra $\psig_t$ on $\pstate$ by
\[
\psig_t:=\sigma\left(\pstate\ni\omega\mapsto \omega(B): B\in \sig_t\right).
\]
We say that a stochastic process $(\ins_t)_{t\in [0,\TT]}$, where $\ins_t:\pstate\to \R$, is \emph{predictable} if the function $(t,\omega)\mapsto \ins_t(\omega)$ is measurable with respect to $\sigma\left(\{(s,t]\times B: s\le t\le \TT, ~B\in \psig_s\}\right)$. Given a predictable nonnegative stochastic process  $\ins:=(\ins_t)_{t\in [0,\TT]}$, such that $\ins_t\le \ubd$ for all $t\in [0,\TT]$, we define the stochastic counting process $(\XX_t^{\ins})_{t\in [0,\TT]}$ by
\begin{equation}
\label{eq:X_def}
\XX_t^{\ins}(\omega)=\omega\left(\left\{(s,x)\in \state: s< t,~ x\le \ins_s(\omega)\right\}\right),\quad\text{(see Figure \ref{fig:X})}.
\end{equation}

Note that $(\XX_t^{\ins})_{t\in [0,\TT]}$ is a non-decreasing integer-valued left-continuous process such that $\XX_t^{\ins}$ is $\psig_t$-measurable for all $t\in [0,\TT]$, and hence $(\XX_t^{\ins})_{t\in [0,\TT]}$ is predictable. In addition, almost-surely, there are only finitely many jumps of $(\XX_t^{\ins})_{t\in [0,\TT]}$, each of which is of size 1. Thus, the process $(\XX_t^{\ins})_{t\in [0,\TT]}$ is a Poisson process with \emph{stochastic intensity} $\ins$. We will work with a specific stochastic intensity $\ins$, namely, we will take the stochastic intensity $\ins^*$ defined by the equation
\begin{equation}
\label{eq:lambda_opt}
\ins^*_t(\omega)=\frac{ \Pheat_{\TT-t}\den(\XX_t^{\ins^*(\omega)}(\omega)+1)}{ \Pheat_{\TT-t}\den(\XX_t^{\ins^*(\omega)}(\omega))}=e^{\der \Den\left(t,\XX_t^{\ins^*(\omega)}(\omega)\right)},
\end{equation}
where we recall \eqref{eq:F_def}. The existence of a solution to \eqref{eq:lambda_opt} was given in \cite[Lemma 4.3]{Klartag_Lehec} via a fixed-point argument, which requires the function $\kk\mapsto \frac{\Pheat_{\TT-t}\den(\kk+1)}{\Pheat_{\TT-t}\den(\kk)}$ to be bounded  for each $t\in [0,\TT]$. In \cite{Klartag_Lehec}, $\den$ itself is assumed to be bounded, which gives the necessary condition, but in our case the log-concavity of $\den$ suffices by Corollary \ref{cor:G_is_bounded}. Note that $\pmeasure$-a.s. $ t \mapsto \ins_t^*$ is continuous except at finitely many points, and in addition, by Corollary \ref{cor:G_is_bounded}, 
\begin{equation}
\label{eq:lambda_M}
\ins_t^*\le  \frac{\den(1)}{\den(0)}=\ubd.
\end{equation}
To ease the notation, from here on, we will denote
\begin{equation}
\label{eq:Xstar_def}
\XX:=(\XX_t)_{t\in [0,\TT]}:=(\XX_t^{\ins^*})_{t\in [0,\TT]} \quad\text{and}\quad \ins:=(\ins_t)_{t\in [0,\TT]}:=(\ins_t^*)_{t\in [0,\TT]}.
\end{equation}
The next lemma provides the time marginals of $\XX$.
\begin{lemma}
\label{lem:X_T}
Let $\XX$ be the process defined by \eqref{eq:Xstar_def}.  For every $t\in [0,\TT]$, the law of $\XX_{t}$ is $(\Pheat_{\TT-t}\den)\pos_t$.
\end{lemma}
\begin{proof}
Let $h:\N\to \R$ be any bounded function and fix $\omega\in \pstate$. By construction, $[0,\TT]\ni 
t\mapsto h(\XX_{t}(\omega))$ is $\pmeasure$-a.s. piecewise constant with jumps of size 1 at $t_1<t_2<\cdots$, where $\omega=\sum_i\delta_{(t_i,z_i)}$. Hence, for each $t\in [0,\TT]$,
\begin{align}
\label{pos_ito}
 h(\XX_{t}(\omega))=h(0)+\int_0^t \der h(\XX_{s}(\omega))\diff \XX_{s}:=h(0)+\sum_{t_i\le t}\der h(\XX_{t_i}(\omega)).
\end{align}
Taking expectation in \eqref{pos_ito}, and applying \cite[Lemma 4.1]{Klartag_Lehec}, we get
\begin{align}
\label{eq:ito_iso}
\EE[ h(\XX_{t})]=h(0)+\EE\left[\int_0^t \der h(\XX_{s})\ins_s\diff s\right].
\end{align}
Let $\prob_t$ be the law of $\XX_{t}$. Differentiating \eqref{eq:ito_iso} in $t$, and using \eqref{eq:lambda_opt}, we can apply summation by parts to get
\begin{align}
\label{eq:ibp}
\sum_{j=0}^{\infty}h(j)\partial_t\prob_t(j)=\sum_{j=0}^{\infty}\der h(j) e^{\der \Den(t,j)}\prob_t(j)=-h(0)e^{\der \Den(t,0)}\prob_t(0)-\sum_{j=0}^{\infty}h(j+1)\der\left[e^{\der \Den(t,j)}\prob_t(j)\right].
\end{align}
Equation \eqref{eq:ibp} holds for all bounded $h$, so fix a non-zero $\kk\in \N$ and let $h(j)=1_{j=\kk}$ to get the discrete Fokker-Planck equation,
\begin{equation}
\label{eq:FK}
\partial_t\prob_t(\kk)=-\der\left[e^{\der \Den(t,\kk-1)}\prob_t(\kk-1)\right]\quad\forall~ k\in\Z_+. 
\end{equation}
To get an equation at $\kk=0$, take $h(j)=1_{j=0}$ and use \eqref{eq:ibp} to deduce
\begin{equation}
\label{eq:FK0}
\partial_t\prob_t(0)=-e^{\der \Den(t,0)}\prob_t(0). 
\end{equation}
Using the convention $e^{\der \Den(t,-1)}=\prob_t(-1)=0$, we can combine \eqref{eq:FK} and \eqref{eq:FK0} to get
\begin{equation}
\label{eq:FK_comb}
\partial_t\prob_t(\kk)=-\der\left[e^{\der \Den(t,\kk-1)}\prob_t(\kk-1)\right]\quad\forall~ k\in\N. 
\end{equation}

One can check that \eqref{eq:FK_comb} is uniquely solved by
\begin{equation}
\label{eq:mu_t}
\prob_t(k)=\Pheat_{\TT-t}\den(k)\pos_t(k) \quad\forall~ k\in\N.
\end{equation}
\end{proof}
An immediate corollary of Lemma \ref{lem:X_T}  is that $\XX_{\TT}$ is distributed as $\prob$. We call the  map $\XX_{\TT}:\pstate\to \N$ the \textbf{Poisson transport map} as it transports $\pmeasure$ to $\prob$. 

\subsection{Properties of the Poisson transport map and ultra-log-concave measures} 

Let us prove a number of properties of the processes $\XX,\ins$, which we will use later.

\begin{lemma}
\label{lem:lambda_martingale}
Let $\XX$ be the process defined by \eqref{eq:Xstar_def}. Then $\ins$ is a  $\pmeasure$-martingale, i.e., the process
\[
[0,\TT]\ni t\mapsto \frac{\Pheat_{\TT-t}\den(\XX_t+1)}{ \Pheat_{\TT-t}\den(\XX_t)}
\]
is a $\pmeasure$-martingale. Further, the common mean of  $\ins$ is $\Pheat_{\TT}\den(1)$. 
\end{lemma}
\begin{proof}
Let  $h:[0,\TT]\times \N\to\R$ be such that the function $[0,\TT]\ni t\mapsto h(t,k)$ is continuous for all $\kk\in \N$. Then the function $[0,\TT]\ni t\mapsto h(t,\XX_t)$ is piecewise absolutely-continuous function in $t$, so 
\begin{equation}
\label{eq:ito_time}
h(t,\XX_t)=h(0,0)+\int_0^t\der h(s,\XX_s)\diff \XX_s+\int_0^t\partial_s h(s,\XX_s)\diff s.
\end{equation}
Take $h(t,\kk):=\frac{\Pheat_{\TT-t}\den(\kk+1)}{ \Pheat_{\TT-t}\den(\kk)}$, and note that it satisfies the continuity condition. Then, 
 using \eqref{eq:F_def}, we get
\begin{align*}
\frac{\Pheat_{\TT-t}\den(\XX_t+1)}{ \Pheat_{\TT-t}\den(\XX_t)}-\frac{\Pheat_{\TT}\den(1)}{ \Pheat_{\TT}\den(0)}&=\int_0^t \der (e^{\der \Den\left(s,\XX_s\right)})\diff \XX_s +\int_0^t\partial_s (e^{\der \Den\left(s,\XX_s\right)})\diff s\\
&=\int_0^t \der (1- \partial_s\Den(s,\XX_s))\diff \XX_s+ \int_0^t\partial_s (e^{\der \Den\left(s,\XX_s\right)})\diff s\\
&=-\int_0^t \der  (\partial_s\Den(s,\XX_s))\diff \XX_s+ \int_0^t\partial_s (e^{\der \Den\left(s,\XX_s\right)})\diff s.
\end{align*}
On the other hand, for every $\kk\in\N$,
\[
\partial_s (e^{\der \Den\left(s,\kk\right)})=e^{\der \Den\left(s,\kk\right)}\partial_s\der\Den(s,\kk)=e^{\der \Den\left(s,\kk\right)}\der\partial_s\Den(s,\kk),
\]
so by \eqref{eq:lambda_opt},
\[
\partial_s (e^{\der \Den\left(s,\XX_s\right)})=\der  (\partial_s\Den(s,\XX_s))\, \ins_s.
\]
We conclude that 
\[
\frac{\Pheat_{\TT-t}\den(\XX_t+1)}{ \Pheat_{\TT-t}\den(\XX_t)}-\frac{\Pheat_{\TT}\den(1)}{ \Pheat_{\TT}\den(0)}=-\int_0^t\der  (\partial_s\Den(s,\XX_s))[\diff \XX_s-\ins_s\diff s].
\]
The process $\left(\XX_t-\int_0^t\ins_s\diff s\right)_{t\in [0,\TT]}$ is called the compensated process, and is a martingale. Hence, the process $\frac{\Pheat_{\TT-t}\den(\XX_t+1)}{ \Pheat_{\TT-t}\den(\XX_t)}$ is a stochastic integral with respect to a martingale, and hence a martingale  \cite[\S 4]{Klartag_Lehec}.

To compute the common mean of $\ins$ note that since $\XX_{\TT}\sim \prob$ (cf. Lemma \ref{lem:X_T}),
\begin{align*}
\EE_{\pmeasure}[\ins_{\TT}]=\EE_{\pmeasure}\left[\frac{\den(\XX_{\TT}+1)}{ \den(\XX_{\TT})}\right]=\sum_{j=0}^{\infty}\frac{\den(j+1)}{\den(j)}\prob(j)=\sum_{j=0}^{\infty}\den(j+1)\pos_{\TT}(j)=\Pheat_{\TT}\den(1).
\end{align*}
\end{proof}
The fact that $\ins$ is a  martingale allows us give a representation of the mean of $\prob$ in terms of the Poisson semigroup, as well as an upper bound.
\begin{corollary}
\label{cor:mean_mu}
\[
\EE[\prob]\overset{(1)}{=}\TT\, \Pheat_{\TT}\den(1)\overset{(2)}{\le }\int_0^{\TT}\frac{\Pheat_{\TT-t}\den(1)}{\Pheat_{\TT-t}\den(0)}\diff t\overset{(3)}{=}\TT-\log \den(0)\overset{(4)}{=}|\log \prob(0)|\overset{(5)}{\le} \frac{\prob(1)}{\prob(0)}.
\]
\end{corollary}
\begin{proof}
To prove identity (1), take $h(j)=j$ and $t=\TT$ in \eqref{eq:ito_iso} to get
\[
\EE[\XX_{\TT}]=\EE\left[\int_0^{\TT} \ins_s\diff s\right]=\TT\EE[\ins_{\TT}]=\TT\,\Pheat_{\TT}\den(1),
\]
where we used Lemma \ref{lem:lambda_martingale}. For the inequality (2), note that $\Pheat_{\TT}f(0)=1$, since $\prob=\den\pos_{\TT}$ is a probability measure, and use Corollary \ref{cor:mono_ratio} to get $\frac{\Pheat_{\TT}\den(1)}{\Pheat_{\TT}f(0)}\le \frac{\Pheat_{\TT-t}\den(1)}{\Pheat_{\TT-t}f(0)}$ for all $t\in [0,\TT]$. For the identity (3), use  \eqref{eq:F_def} to compute
\begin{align*}
\int_0^{\TT}\frac{\Pheat_{\TT-t}\den(1)}{\Pheat_{\TT-t}f(0)}\diff t=\int_0^{\TT} e^{\der \Den(t,0)}\diff t=\int_0^{\TT} [1-\partial_t\Den(t,0)]\diff t=\TT-[\Den (\TT,0)-\Den (0,0)].
\end{align*}
The result follows since $\Den (\TT,0)=\log \den (0)$, and $\Den (0,0)=\log \Pheat_{\TT}\den(0)=0$ (because $\Pheat_{\TT}f(0)=1$ as $\prob=\den\pos_{\TT}$ is a probability measure). The identity (4) follows from $\prob=\den\pos_{\TT}$. Finally, the inequality (5) holds since, by  Corollary \ref{cor:mono_ratio}, $\frac{\Pheat_{\TT-t}\den(1)}{\Pheat_{\TT-t}f(0)}\le \frac{\den(1)}{\den(0)}$ so, by (3)-(4),
\[
|\log \prob(0)|=\int_0^{\TT}\frac{\Pheat_{\TT-t}\den(1)}{\Pheat_{\TT-t}\den(0)}\diff t\le \TT \frac{\den(1)}{\den(0)}=\frac{\prob(1)}{\prob(0)}.
\]
\end{proof}

\section{Contraction of the Poisson transport map}
\label{sec:contraction}
The main result of this section is that the Poisson transport map is a contraction when the target measures are ultra-log-concave. This result will follow from the following more general theorem, showing that the Malliavin derivative of $\XX$ is binary and nonnegative.

\begin{theorem}
\label{thm:contraction}
Fix a real number $\TT>0$ and let $\prob=\den\pos_{\TT}$ be an ultra-log-concave probability measure over $\N$. Let $\XX$ be the process defined by \eqref{eq:Xstar_def}.   Then, $\pmeasure$-almost-surely, for every $s\in [0,\TT]$,
\begin{equation}
\label{eq:contraction}
\derm_{(t,z)}\XX_s\in \{0,1\}\quad\forall~ (t,z)\in \state. 
\end{equation}
\end{theorem}

An immediate corollary of Theorem \ref{thm:contraction} is that the Poisson transport map is a contraction, thus proving Theorem \ref{thm:contraction_intro}.
\begin{corollary}
\label{cor:contraction}
Fix a real number $\TT>0$ and let $\prob=\den\pos_{\TT}$ be an ultra-log-concave probability measure over $\N$. Let $\XX_{\TT}$ be the Poisson transport map from $\pmeasure$ to $\prob$. Then, $\pmeasure$-almost-surely,
\begin{equation}
\label{eq:contraction_cor}
\derm_{(t,z)}\XX_{\TT}\in \{0,1\}\quad\forall~ (t,z)\in \state. 
\end{equation}
\end{corollary}

Let us turn to the proof of Theorem \ref{thm:contraction}.

\begin{proof}[Proof of Theorem \ref{thm:contraction}]
Fix $(t,z)\in  \state$ and $\omega\in\pstate$. Then $\pmeasure$-a.s., there exists $n\in\N$ such that $\omega=\sum_{i=1}^n\delta_{(t_i,z_i)}$ for $(t_i,z_i)\in \state$,  with $i\in [n]:=\{1,\ldots, n\}$, $0<t_1<\cdots<t_n<\TT$, and $t\neq t_i$ for all $i\in [n]$. Fix $s\in [0,\TT]$. We need to show that
\begin{equation}
\label{eq:contraction_proof}
\derm_{(t,z)}\XX_s(\omega)=\XX_s(\omega+\delta_{(t,z)})-\XX_s(\omega)\in \{0,1\}.
\end{equation}
Let us first explain the intuition why \eqref{eq:contraction_proof} holds, and then turn to its rigorous verification. There are  three cases to consider. The first two are easy, and the third one is the interesting one. 
\begin{itemize}
\item \textbf{Case 1}. $s\le t$: Then the contribution of  the atom $(t,z)$ is not captured by either $\XX_s(\omega+\delta_{(t,z)})$ or $\XX_s(\omega)$, so both processes behave identically, and hence $\derm_{(t,z)}\XX_s(\omega)=0$.
\item \textbf{Case 2}. $t<s$ and $z$ lies above the curve $\ins(\omega+\delta_{(t,z)})$: Then the atom $(t,z)$ is not counted by the process $\XX(\omega+\delta_{(t,z)})$, so the processes $\XX(\omega+\delta_{(t,z)})$ and $\XX(\omega)$ are equal, and hence $\derm_{(t,z)}\XX_s(\omega)=0$.
\item  \textbf{Case 3}. The interesting case is $t<s$ and $z$ lies below the curve $\ins(\omega+\delta_{(t,z)})$, so the processes $\XX(\omega+\delta_{(t,z)})$ and $\XX(\omega)$ can in fact differ. Our goal is show that when the two processes differ, $\XX(\omega+\delta_{(t,z)})$ is always greater than $\XX(\omega)$, but by no more than 1. The key to prove this is to use the log-concavity of $\den$. Using the explicit expression of $\ins$ \eqref{eq:lambda_opt},  we can reason about the relation between $\ins(\omega+\delta_{(t,z)})$ and $\ins(\omega)$, and hence about the relation between $\XX(\omega+\delta_{(t,z)})$ and $\XX(\omega)$. 
\end{itemize}
Let us now turn to the actual proof of the theorem. \\

\noindent \textbf{Case 1}. $s\le t$:  We will show
\begin{equation}
\label{eq:case1}
\derm_{(t,z)}\XX_s(\omega)=\XX_s(\omega+\delta_{(t,z)})-\XX_s(\omega)=0.
\end{equation}
From the definition of  $\XX(\omega+\delta_{(t,z)})$, we know that the atom $(t,z)$ is not counted by $\XX(\omega+\delta_{(t,z)})$. So to verify \eqref{eq:case1} it suffices to show that each atom $(t_i,z_i)$ is either counted by both $\XX(\omega+\delta_{(t,z)})$ and $\XX(\omega)$, or by neither. If $t_i\ge s$  for all $i\in [n]$, then \eqref{eq:case1} holds since both $\XX(\omega+\delta_{(t,z)})$ and $\XX(\omega)$ start at 0, and neither adds any atom by time $s$. 

If there exists $i\in [n]$ such that $t_i<s$, let us denote $i_{\max}:=\max\{i\in [n]:t_i<s\}$. Since the processes are left-continuous, starting at 0, $\XX_{t_1}(\omega+\delta_{(t,z)})=\XX_{t_1}(\omega)=0$. Hence, by \eqref{eq:lambda_opt},
\[
\ins_{t_1}(\omega)=\frac{ \Pheat_{\TT-t_1}\den(\XX_{t_1}(\omega)+1)}{ \Pheat_{\TT-t_1}\den(\XX_{t_1}(\omega))}=\frac{ \Pheat_{\TT-t_1}\den(\XX_{t_1}(\omega+\delta_{(t,z)})+1)}{ \Pheat_{\TT-t_1}\den(\XX_{t_1}(\omega+\delta_{(t,z)}))}=\ins_{t_1}(\omega+\delta_{(t,z)}).
\]
It follows that
\begin{equation}
\label{eq:case1_lambda_lambda+}
  z_1\le \ins_{t_1}(\omega+\delta_{(t,z)})\quad\Longleftrightarrow\quad z_1\le \ins_{t_1}(\omega).
\end{equation}
Hence, for each $r\in(t_1,t_2\wedge s]$, (if $n=1$ then for each $r\in (t_1,s]$), $\XX_{r}(\omega+\delta_{(t,z)})=\XX_{r}(\omega)$.  If $i_{\max}=1$, we are done. Otherwise, if $i_{\max}\ge 2$, we can repeat the above argument inductively for $i\in \{2,\ldots,i_{\max}\}$ to conclude that \eqref{eq:case1} holds. \\

\noindent \textbf{Case 2}. $t<s$ and $z>\ins_t(\omega+\delta_{(t,z)})$:  We will show
\begin{equation}
\label{eq:case2}
\derm_{(t,z)}\XX_s(\omega)=\XX_s(\omega+\delta_{(t,z)})-\XX_s(\omega)=0.
\end{equation}
The argument of Case 1 shows that $\XX_t(\omega+\delta_{(t,z)})=\XX_t(\omega)$. Since $z>\ins_t(\omega+\delta_{(t,z)})$, the atom $(t,z)$ is not counted by $\XX(\omega+\delta_{(t,z)})$.  It remains to analyze the atoms $(t_i,z_i)$ for $i\in [n]$. If there exist no $t_i$ such that $t<t_i<s$, then it is clear that $\XX_s(\omega+\delta_{(t,z)})=\XX_s(\omega)$, so \eqref{eq:case2} holds. Suppose then that there exist $t_i$ such that $t<t_i<s$, and let $i_{\min}:=\min\{i\in [n]:t<t_i<s\}$. Similar to Case 1, for $r\in (t,t_{i_{\min}}]$ we have $\XX_r(\omega+\delta_{(t,z)})=\XX_r(\omega)$. In particular,  $\XX_{t_{i_{\min}}}(\omega+\delta_{(t,z)})=\XX_{t_{i_{\min}}}(\omega)$ so, as in Case 1,
\begin{equation}
\label{eq:case2_lambda_lambda+}
  z_{i_{\min}}\le \ins_{t_{i_{\min}}}(\omega+\delta_{(t,z)})\quad\Longleftrightarrow\quad z_{i_{\min}}\le \ins_{t_{i_{\min}}}(\omega).
\end{equation}
Hence, for each $r\in (t_{i_{\min}},t_{i_{\min}+1}\wedge s]$, (if $i_{\min}=n$ then for $r\in (t_{i_{\min}}, s]$), $\XX_{r}(\omega+\delta_{(t,z)})=\XX_{r}(\omega)$. We may repeat the argument above inductively for all $i\in [n]$ satisfying $t<t_i<s$ to conclude that  $\XX_{s}(\omega+\delta_{(t,z)})=\XX_{s}(\omega)$, so \eqref{eq:case2} holds.\\

\noindent \textbf{Case 3}. $t<s$ and $z\le \ins_t(\omega+\delta_{(t,z)})$: In contrast to Cases 1 and 2 we will show that
\begin{equation}
\label{eq:case3}
\derm_{(t,z)}\XX_s(\omega)=\XX_s(\omega+\delta_{(t,z)})-\XX_s(\omega)\in \{0,1\}.
\end{equation}
Again, the argument of Case 1 shows that $\XX_t(\omega+\delta_{(t,z)})=\XX_t(\omega)$. In contrast to Case 2, since $z\le \ins_t(\omega+\delta_{(t,z)})$, the atom $(t,z)$ is counted by $\XX(\omega+\delta_{(t,z)})$. Let us analyze the possible values of $\XX_s(\omega+\delta_{(t,z)})$ and $\XX_s(\omega)$. If there exist no $t_i$ such that $t<t_i<s$, then 
\[
\XX_s(\omega+\delta_{(t,z)})=\XX_t(\omega+\delta_{(t,z)})+1\quad\text{and}\quad \XX_s(\omega)=\XX_t(\omega),
\]
so $\derm_{(t,z)}\XX_s(\omega)=1$, and hence \eqref{eq:case3} holds. 

Suppose then that there exists $t_i$ such that $t<t_i<s$, and as in Case 2, let $i_{\min}:=\min\{i\in [n]:t<t_i<s\}$. For $r\in (t,t_{i_{\min}}]$, we have
\[
\XX_r(\omega+\delta_{(t,z)})=\XX_t(\omega)+1\quad\text{and}\quad \XX_r(\omega)=\XX_t(\omega).
\]
In particular,
\[
\XX_{t_{i_{\min}}}(\omega+\delta_{(t,z)})=\XX_{t_{i_{\min}}}(\omega)+1,
\]
so by Proposition \ref{prop:poisson_presev}, and \eqref{eq:lc_def_ulc}, 
\begin{align}
\label{eq:lambda_lc}
\begin{split}
\ins_{t_{i_{\min}}}(\omega+\delta_{(t,z)})&=\frac{ \Pheat_{\TT-t_{i_{\min}}}\den(\XX_{t_{i_{\min}}}(\omega+\delta_{(t,z)})+1)}{ \Pheat_{\TT-t_{i_{\min}}}\den(\XX_{t_{i_{\min}}}(\omega+\delta_{(t,z)}))}=\frac{ \Pheat_{\TT-t_{i_{\min}}}\den(\XX_{t_{i_{\min}}}(\omega)+2)}{ \Pheat_{\TT-t_{i_{\min}}}\den(\XX_{t_{i_{\min}}}(\omega)+1)}\\
&\le \frac{ \Pheat_{\TT-t_{i_{\min}}}\den(\XX_{t_{i_{\min}}}(\omega)+1)}{ \Pheat_{\TT-t_{i_{\min}}}\den(\XX_{t_{i_{\min}}}(\omega))}=\ins_{t_{i_{\min}}}(\omega).
\end{split}
\end{align}
Let us record then the one-direction analogue of \eqref{eq:case2_lambda_lambda+},
\begin{equation}
\label{eq:case3_lambda_lambda+}
z_{i_{\min}}\le\ins_{t_{i_{\min}}}(\omega+\delta_{(t,z)})\quad\Longrightarrow \quad z_{i_{\min}}\le\ins_{t_{i_{\min}}}(\omega).
\end{equation}

We now have a three sub-cases to consider. \\

\noindent \textbf{Case 3.1}. $z_{i_{\min}}\le \ins_{t_{i_{\min}}}(\omega+\delta_{(t,z)})$: Applying \eqref{eq:case3_lambda_lambda+} we can deduce that for each $r\in(t_{i_{\min}},t_{i_{\min}+1}\wedge s]$, (if $i_{\min}=n$ then for each  $r\in(t_{i_{\min}}, s]$), $\derm_{(t,z)}\XX_r(\omega)=1$, since both $\XX(\omega)$ and $\XX(\omega+\delta_{(t,z)})$ count $(t_{i_{\min}},z_{i_{\min}})$.\\

\noindent \textbf{Case 3.2}. $z_{i_{\min}}>\ins_{t_{i_{\min}}}(\omega+\delta_{(t,z)})$ and  $z_{i_{\min}}>\ins_{t_{i_{\min}}}(\omega)$: For each $r\in(t_{i_{\min}},t_{i_{\min}+1}\wedge s]$, (if $i_{\min}=n$ then for each  $r\in(t_{i_{\min}}, s]$), we have $\derm_{(t,z)}\XX_r(\omega)=1$, since both $\XX(\omega)$ and $\XX(\omega+\delta_{(t,z)})$  did not count $(t_{i_{\min}},z_{i_{\min}})$.\\

\noindent \textbf{Case 3.3}. $z_{i_{\min}}>\ins_{t_{i_{\min}}}(\omega+\delta_{(t,z)})$ and  $z_{i_{\min}}\le \ins_{t_{i_{\min}}}(\omega)$: For each $r\in(t_{i_{\min}},t_{i_{\min}+1}\wedge s]$, (if $i_{\min}=n$ then for each  $r\in(t_{i_{\min}}, s]$), we have $\derm_{(t,z)}\XX_r(\omega)=0$, since  $\XX(\omega)$ counted $(t_{i_{\min}},z_{i_{\min}})$, but $\XX(\omega+\delta_{(t,z)})$  did not. 

If there exists no $t_i$, for $i>i_{\min}$, such that $t<t_i<s$, then Cases 3.1-3.3 verify \eqref{eq:case3}. Suppose then that there exist  $i>i_{\min}$ such that $t<t_i<s$. We will proceed inductively. From Cases 3.1-3.3 we have that $\derm_{(t,z)}\XX_{t_{i_{\min}+1}}(\omega)\in\{0,1\}$. If $\derm_{(t,z)}\XX_{t_{i_{\min}+1}}(\omega)=1$, then arguing as in \eqref{eq:lambda_lc}, we have
\begin{equation}
\label{eq:lambda_lc+1}
\ins_{t_{i_{\min}+1}}(\omega+\delta_{(t,z)})\le \ins_{t_{i_{\min}+1}}(\omega).
\end{equation}
We now repeat Cases 3.1-3.3, replacing $i_{\min}$ by $i_{\min}+1$.  If $\derm_{(t,z)}\XX_{t_{i_{\min}+1}}(\omega)=0$, then $\ins_{t_{i_{\min}+1}}(\omega+\delta_{(t,z)})= \ins_{t_{i_{\min}+1}}(\omega)$, so \eqref{eq:lambda_lc+1} holds, and again  we repeat Cases 3.1-3.3, replacing $i_{\min}$ by $i_{\min}+1$. Continuing in this manner we deduce \eqref{eq:case3}.
\end{proof}

The proof of Theorem \ref{thm:contraction} yields the following necessary condition for the Malliavin derivative being 1.
\begin{corollary}
\label{cor:Malliavin_necessary}
Fix $(t,z)\in  \state$ and $\omega\in\pstate$. Then $\pmeasure$-a.s., given $s\in [0,\TT]$, if $\derm_{(t,z)}\XX_s(\omega)=1$, we must have $t<s$ and $z\le \ins_t(\omega+\delta_{(t,z)})= \ins_t(\omega)$.
\end{corollary}
\begin{proof}
The conditions $t<s$ and $z\le \ins_t(\omega+\delta_{(t,z)})$ hold because the proof of Theorem \ref{thm:contraction} showed that $\derm_{(t,z)}\XX_s(\omega)=0$ in Cases 1-2. The condition $\ins_t(\omega+\delta_{(t,z)})= \ins_t(\omega)$ holds since in Case 3 we have shown $\XX_t(\omega+\delta_{(t,z)})=\XX_t(\omega)$, so the result follows from \eqref{eq:lambda_opt}. 
\end{proof}

\subsection{The Brownian transport map vs. the Poisson transport map}
\label{subsec:BrownianVsPoisson}
Let us elaborate on the similarities and dissimilarities between the Brownian transport map \cite{mikulincer2021brownian} and the Poisson transport map. For  simplicity, let us take $\TT=1$. We begin with a  sketch of the Brownian transport map. Let $\prob$ be a probability measure on $\R^d$ of the form $\prob=\den\gamma_d$. Denote by $(\PheatE_t)_{t\in [0,1]}$ the heat semigroup on $\R^d$, and consider the stochastic differential equation,
\begin{equation}
\label{eq:Follmer_SDE}
\diff \YY_t=\nabla\log \PheatE_{1-t}\den(\YY_t)\diff t+\diff B_t,\quad \YY_0=0,
\end{equation}
where $(B_t)_{t\in [0,1]}$ is a standard Brownian motion in $\R^d$. The process $\YY:=(\YY_t)_{t\in [0,1]}$ is known as the \emph{F\"ollmer process} \cite{follmer2005entropy, follmer2006time, lehec2013representation}, and can be seen as Brownian motion conditioned to be distributed like $\prob$ at time 1. Alternatively, the process $\YY$ is a solution to an entropy minimization problem  over the Wiener space. In \cite{mikulincer2021brownian}, $\YY_1$ is called the \textbf{Brownian transport map}, as it transports the Wiener measure (the law of $(B_t)_{t\in [0,1]}$) onto $\prob$. 

Now suppose that $\prob=\den\gamma_d$ is such that $\den:\R^d\to \R_{\ge 0}$ is log-concave. It was shown in \cite[Theorem 1.1]{mikulincer2021brownian} that, in such setting, the Brownian transport map $\YY_1$ is 1-Lipschitz, in the sense that the Malliavin derivative $\derm$ of $\YY_1$ is bounded in absolute value by 1. The proof of this result proceeds by differentiating  \eqref{eq:Follmer_SDE} (with $\derm$) to get \cite[Proposition 3.10]{mikulincer2021brownian},
\begin{equation}
\label{eq:Follmer_SDE_Malliavin}
\partial_s\derm\YY_s=\nabla^2\log \PheatE_{1-s}\den(\YY_s)\derm\YY_s.
\end{equation}
Hence, to show that $\YY_1$ is 1-Lipschitz, i.e., $|\derm\YY_1|\le 1$, it suffices to control $\nabla^2\log \PheatE_{1-s}\den(\YY_s)$, and then use Gr\"onwall's inequality. In particular, when $\den$ is log-concave, $\PheatE_{1-s}\den$ is also log-concave (consequence of the   Pr\'ekopa-Leindler inequality), i.e.,
\begin{equation}
\label{eq:heat_lc}
\nabla^2\log \PheatE_{1-s}\den(\YY_s)\le 0,
\end{equation}
so \eqref{eq:Follmer_SDE_Malliavin} and Gr\"onwall's inequality yield  $|\derm\YY_1|\le 1$.

The analogue in the discrete setting of the F\"ollmer process \eqref{eq:Follmer_SDE} is the process $\XX$ defined in \eqref{eq:Xstar_def}. Indeed, it was shown in \cite{Klartag_Lehec} that the process $\XX$ is the solution to the corresponding entropy minimization problem on the Poisson  space. Unlike the continuous setting, here we do not have an analogue of \eqref{eq:Follmer_SDE_Malliavin}, but we do have an analogue of \eqref{eq:heat_lc}. The process $\ins_s= \frac{\Pheat_{1-s}\den(\XX_s+1)}{ \Pheat_{1-s}\den(\XX_s)}$ plays the role of $\nabla\log \PheatE_{1-s}\den(\YY_s)$, and the next result is the analogue of \eqref{eq:heat_lc}.

\begin{lemma}
\label{lem:heat_lc_Poison}
For every $s\in [0,1]$, $\pmeasure$-almost-surely,
\[
\derm_{(t,z)}\ins_s\le 0\quad\forall~ (t,z)\in \state.
\]
\end{lemma}
\begin{proof}
Fix $\omega\in\pstate$. By definition,
\begin{align*}
\derm_{(t,z)}\ins_s(\omega)=\ins_s(\omega+\delta_{(t,z)})-\ins_s(\omega)=\frac{\Pheat_{1-s}\den(\XX_s(\omega+\delta_{(t,z)})+1)}{ \Pheat_{1-s}\den(\XX_s(\omega+\delta_{(t,z)}))}-\frac{\Pheat_{1-s}\den(\XX_s(\omega)+1)}{ \Pheat_{1-s}\den(\XX_s(\omega))}.
\end{align*}
By Theorem \ref{thm:contraction}, $\XX_s(\omega+\delta_{(t,z)})\in \{\XX_s(\omega),\XX_s(\omega)+1\}$. If $\XX_s(\omega+\delta_{(t,z)})=\XX_s(\omega)$, then $\derm_{(t,z)}\ins_s(\omega)=0$. If $\XX_s(\omega+\delta_{(t,z)})=\XX_s(\omega)+1$, then
\[
\derm_{(t,z)}\ins_s(\omega)=\frac{\Pheat_{1-s}\den(\XX_s(\omega)+2)}{ \Pheat_{1-s}\den(\XX_s(\omega)+1)}-\frac{\Pheat_{1-s}\den(\XX_s(\omega)+1)}{ \Pheat_{1-s}\den(\XX_s(\omega))}\le 0,
\]
where the inequality holds by Proposition  \ref{prop:poisson_presev} and \eqref{eq:lc_def_ulc}. 
\end{proof} 
Let us conclude with a few remarks on some other differences between the Brownian and Poisson transport maps.
\begin{remark}
\label{rem:BM-Poisson}
$~$

\begin{enumerate}
\item In the Brownian transport map setting, the source measure is always the Wiener measure on the Wiener space $C([0,1])$ of continuous functions $[0,1]\to \R$, independent of the target measure $\prob$.  In contrast, the transported Poisson measure $\pmeasure$ depends on $\prob$, because the space $\state=[0,\TT]\times [0,\ubd]$ depends on $\prob$ via $\ubd$. This difference is not material since the functional inequalities satisfied by $\pmeasure$ do not depend on $\ubd$. 

\item The fact that the Brownian transport map is 1-Lipschitz, when $\prob$ is more log-concave than the Gaussian, means that the functional inequalities which hold for the Gaussian also hold for $\prob$ \emph{with the same constants}. In contrast, the constants in the functional inequalities for ultra-log-concave measures $\prob=\den\pos_{\TT}$, obtained from the Poisson transport map, are different from those satisfied by $\pos_{\TT}$. This is not a deficiency of the Poisson transport map, but rather a manifestation of the discrete nature of the probability measures under consideration.
\end{enumerate}
\end{remark}

\section{Functional inequalities}
\label{sec:func_inq}
In this section we show how Corollary \ref{cor:contraction} can be used to deduce functional inequalities for ultra-log-concave measures. In particular, the results of this section verify Theorem \ref{thm:modified_LSI_ULC_into}, Theorem \ref{thm:PhiSoboleUlc_intro}, and Theorem \ref{thm:entropy-transport_ULC_intro}. The proofs of all of the results below proceed by using an appropriate  functional inequality for $\pmeasure$ (cf. Section \ref{subsec:poisson_space}), and then, using Corollary \ref{cor:contraction}, transporting these inequalities to ultra-log-concave measures. 

\subsection{$\Phi$-Sobolev inequalities} 
\label{subsec:Phi_Sobolev}
In this section we prove both Theorem \ref{thm:PhiSoboleUlc_intro} and Theorem \ref{thm:modified_LSI_ULC_into}.  

\begin{definition}
Let $\mathcal I\subseteq \R$ be a closed interval, not necessarily bounded, and let $\Phi:\mathcal I\to \R$ be a smooth convex function. Let $(E,\mathcal E,Q)$ be a probability Borel space. The $\Phi$-entropy functional $\Ent^{\Phi}_Q$ is defined on the set of $Q$-integrable functions $G:(E,\mathcal E)\to (\mathcal I,\mathcal B(\mathcal I))$, where $\mathcal B(\mathcal I)$ stands for the Borel sigma-algebra of  $\mathcal I$,   by
\begin{equation}
\label{eq:Phi_Ent_def}
\Ent^{\Phi}_Q(G):=\int_E\Phi(G)\diff Q-\Phi\left(\int_E G \diff Q\right).
\end{equation}
\end{definition} 

As shown by Chafa\"{\i}, the Poisson measure $\pmeasure$ satisfies $\Phi$-Sobolev inequalities:

\begin{theorem}{\cite[Eq. (61)]{MR2081075}.}
\label{thm:PhiSobolePoisson}
Let $\mathcal I\subseteq \R$ be a closed interval, not necessarily bounded, and let $\Phi:\mathcal I\to \R$ be a smooth convex function. Suppose that the function
\[
\{(u,v)\in \R^2:(u,u+v)\in \mathcal I\times\mathcal I\}\ni (u,v)\quad\mapsto \quad \Psi(u,v):=\Phi(u+v)-\Phi(u)-\Phi'(u)v
\]
is nonnegative and convex. Then, for any $G\in L^2(\pstate,\pmeasure)$, such that $\pmeasure$-a.s. $G,G+\derm G\in \mathcal I$,
\begin{equation}
\label{eq:PhiSobolevPoisson}
\Ent^{\Phi}_{\pmeasure}(G)\le \EE_{\pmeasure}\left[\int_{\state}\Psi(G,\derm_{(t,z)}G)\diff t\diff z\right].
\end{equation}
\end{theorem}

Let us now transport the inequality \eqref{eq:PhiSobolevPoisson} to ultra-log-concave measures, using the Poisson transport map, thus proving Theorem \ref{thm:PhiSoboleUlc_intro}.

\begin{theorem}
\label{thm:PhiSoboleUlc}
Let $\prob$ be an ultra-log-concave probability measure over $\N$. Let $\mathcal I\subseteq \R$ be a closed interval, not necessarily bounded, and let $\Phi:\mathcal I\to \R$ be a smooth convex function. Suppose that the function
\[
\{(u,v)\in \R^2:(u,u+v)\in \mathcal I\times\mathcal I\}\ni (u,v)\quad\mapsto \quad \Psi(u,v):=\Phi(u+v)-\Phi(u)-\Phi'(u)v
\]
is nonnegative and convex. Then, for any $g\in L^2(\N,\prob)$, such that $\prob$-a.s. $g,g+\der g\in\mathcal I$,
\begin{equation}
\label{eq:PhiSobolevUlc}
\Ent^{\Phi}_{\prob}(g)\le \LL\,\EE_{\prob}[\Psi(g,\der g)].
\end{equation}
\end{theorem}

\begin{proof}
Define $G\in L^2(\pstate,\pmeasure)$ by $G(\omega):=g(\XX_T(\omega))$, and apply \eqref{eq:PhiSobolevPoisson} to get
\begin{align}
\label{eq:Poisson_Phi_Sob}
\begin{split}
\Ent^{\Phi}_{\prob}(g)=\Ent^{\Phi}_{\pmeasure}(G)&\le \EE_{\pmeasure}\left[\int_{\state}\Psi(G,\derm_{(t,z)}G)\diff t\diff z\right]\\
&=\EE_{\pmeasure}\left[\int_{\state}\Psi(g\circ \XX_T,(\der g\circ \XX_{\TT})\cdot \derm_{(t,z)}\XX_{\TT})\diff t\diff z\right],
\end{split}
\end{align}
where the last equality holds by Corollary \ref{cor:contraction} and Lemma \ref{lem:chain_rule}.  Since $\derm_{(t,z)}\XX_{\TT}\in \{0,1\}$ by Corollary \ref{cor:contraction}, we have that $\pmeasure$-a.s.,
\begin{equation}
\label{eq:Psi_indicator}
\Psi(g\circ \XX_T,(\der g\circ \XX_{\TT})\cdot \derm_{(t,z)}\XX_{\TT})=\Psi(g\circ \XX_T,(\der g\circ \XX_{\TT}))1_{\{\derm_{(t,z)}\XX_{\TT}=1\}}.
\end{equation}
On the other hand, by Corollary \ref{cor:Malliavin_necessary}, we have, $\pmeasure$-a.s., $1_{\{\derm_{(t,z)}\XX_{\TT}=1\}}\le 1_{\{z\le \ins_t\}}$. Since $\Psi$ is nonnegative, we conclude from \eqref{eq:Psi_indicator} that 
\begin{equation}
\label{eq:Psi_indicator_inq}
\Psi(g\circ \XX_T,(\der g\circ \XX_{\TT})\cdot \derm_{(t,z)}\XX_{\TT})\le \Psi(g\circ \XX_T,(\der g\circ \XX_{\TT}))  1_{\{z\le \ins_t\}}.
\end{equation}
It follows from \eqref{eq:Poisson_Phi_Sob} and \eqref{eq:Psi_indicator_inq} that
\begin{align*}
\Ent^{\Phi}_{\prob}(g)&\le \EE_{\pmeasure}\left[\int_{\state}\Psi(g\circ \XX_T,\der g\circ \XX_{\TT})1_{\{z\le \ins_t\}}\diff t\diff z\right]\\
&=\EE_{\pmeasure}\left[\Psi(g\circ \XX_T,\der g\circ \XX_{\TT})\int_{\state}1_{\{z\le \ins_t\}}\diff t\diff z\right]\\
&=\EE_{\pmeasure}\left[\Psi(g\circ \XX_T,\der g\circ \XX_{\TT})\int_0^{\TT} \ins_t \diff t\right].
\end{align*}
By \eqref{eq:lambda_opt}, $\int_0^{\TT} \ins_t \diff t=\int_0^{\TT} \frac{\Pheat_{\TT-t}\den(\XX_t+1)}{\Pheat_{\TT-t}f(\XX_t)}\diff t$. On the other hand, $\frac{\Pheat_{\TT-t}\den(\XX_t+1)}{\Pheat_{\TT-t}f(\XX_t)}\le \frac{\Pheat_{\TT-t}\den(1)}{\Pheat_{\TT-t}f(0)}$ by Proposition \ref{prop:poisson_presev} and \eqref{eq:lc_def_ulc_mono}. The proof is complete by Corollary \ref{cor:mean_mu}(3).
\end{proof}
Taking $\Phi(r)=r\log r$ we deduce a modified logarithmic Sobolev inequality, thus proving Theorem \ref{thm:modified_LSI_ULC_into}.
\begin{corollary}
\label{cor:modified_LSI_ULC}
Let $\prob$ be an ultra-log-concave probability measure over $\N$. Then, for any positive $g\in L^2(\N,\prob)$,
\[
\Ent_{\prob}(g)\le \LL\,\EE_{\prob}[\Psi(g,\der g)],
\]
where $\Psi(u,v):=(u+v)\log(u+v)-u\log u-(\log u+1)v$.
\end{corollary}

\subsection{Transport-entropy inequalities}
\label{subsec:entropy-transport}
In this section we prove Theorem \ref{thm:entropy-transport_ULC_intro}. We fix an ultra-log-concave probability measure $\prob=\den\pos_{\TT}$ on $\N$, and recall the definition of the associated Poisson space from Section \ref{subsec:poisson_space}. The starting point is a transport-entropy inequality for the Poisson measure $\pmeasure$ by Ma, Shen, Wang, and Wu (a special case of their more general result), which requires the following definitions. Let $d$ be the total variation distance on $\pstate$ given by $d(\omega,\omega'):=|\omega-\omega'|(\state)$ \cite[Remark 2.4]{MR2797986}. Given two probability measures $Q,P$ on $(\pstate,\psig)$, with finite first moments, let the Wasserstein 1-distance between them be given by
\[
W_{1,d}(Q,P):=\inf_{\Pi}\int_{\pstate\times\pstate}d(\omega,\omega')\diff \Pi(\omega,\omega'),
\]
where the infimum is taken over all couplings $\Pi$ of $(Q,P)$. If $Q$ is absolutely continuous with respect to $P$, let the relative entropy between them be
\[
H(Q|P):=\int_{\pstate} \log\left(\frac{\diff Q}{\diff P}\right)\diff Q.
\]
Finally, given $c>0$, let 
\[
\alpha_c(r):=c\left[\left(1+\frac{r}{c}\right)\log\left(1+\frac{r}{c}\right)-\frac{r}{c}\right].
\]
\begin{theorem}{\cite[Eq. (2.4)]{MR2797986}.}
\label{thm:transport_Poisson}
For any probability measure $Q$ on $(\pstate,\psig)$ which is absolutely continuous with respect to $\pmeasure$, and has a finite first  moment, we have
\begin{equation}
\label{eq:transport_inq_Poisson}
\alpha_{\TT\ubd}\left(W_{1,d}(Q,\pmeasure)\right)\le H(Q|\pmeasure),
\end{equation}
where $\ubd=\frac{\den(1)}{\den(0)}$. 
\end{theorem}

Let us now transport the inequality \eqref{eq:transport_inq_Poisson}, thus proving Theorem \ref{thm:entropy-transport_ULC_intro}. To do so, we define the Wasserstein 1-distance between two probability measures $\nu,\rho$ on $\N$, with finite first moments, by
\[
W_{1,|\cdot|}(\nu,\rho):=\inf_{\Pi}\int_{\N\times\N}|x-y|\diff\Pi(x,y),
\]
where the infimum is taken over all couplings $\Pi$ of $(\nu,\rho)$.
\begin{theorem}
\label{thm:entropy-transport_ULC}
Let $\prob=\den\pos_{\TT}$ be an ultra-log-concave probability measure on $\N$ with $\ubd=\frac{\den(1)}{\den(0)}$. Then,  for any probability measure $\nu$ on $\N$ which is absolutely continuous with respect to $\prob$, and has a finite first  moment, we have 
\begin{equation}
\label{eq:transport_inq_ULC}
\alpha_{\TT\ubd}\left(W_{1,|\cdot|}(\nu,\prob)\right)\le H(\nu|\prob).
\end{equation}
\end{theorem}
\begin{proof}
We follow the proof of \cite[Lemma 2.1]{MR2078555}. Fix a probability measure $\nu$ on $\N$ which is absolutely continuous with respect to $\prob$, and has a finite first  moment. By \cite[Eq. (2.1)]{MR2078555},
\begin{equation}
\label{eq:entropy_rep}
H(\nu|\prob)=\inf_Q\{H(Q|\pmeasure): Q\circ \XX_{\TT}^{-1}=\nu\}.
\end{equation}
Hence, by \eqref{eq:transport_inq_Poisson}, it suffices to show that
\begin{equation}
\label{eq:infW_W}
\alpha_{\TT\ubd}\left(W_{1,|\cdot|}(\nu,\prob)\right)\le \inf_Q\left\{\alpha_{\TT\ubd}\left(W_{1,d}(Q,\pmeasure)\right):Q\circ \XX_{\TT}^{-1}=\nu\right\}.
\end{equation}
Since $\alpha_{\TT\ubd}$ is monotonic, \eqref{eq:infW_W} is equivalent to 
\begin{equation}
\label{eq:infW_W1}
W_{1,|\cdot|}(\nu,\prob)\le \inf_Q\left\{W_{1,d}(Q,\pmeasure):Q\circ \XX_{\TT}^{-1}=\nu\right\}.
\end{equation}
To establish \eqref{eq:infW_W1}, note that by Corollary \ref{cor:contraction}, and \cite[Lemma 2.3]{MR2797986}, we have that $\XX_{\TT}:(\pstate,d)\to (\N,|\cdot|)$ is 1-Lipschitz. Fix $Q$ such that  $Q\circ \XX_{\TT}^{-1}=\nu$, and let $\Pi$ be the coupling attaining the minimum in the definition of $W_{1,d}(Q,\pmeasure)$. Note that $\Pi\circ \XX_{\TT}^{-1}$ is a coupling of $(Q\circ \XX_{\TT}^{-1},\pmeasure\circ \XX_{\TT}^{-1})=(\nu,\prob)$. Hence,
\begin{align*}
W_{1,|\cdot|}(\nu,\prob) \le \int_{\pstate\times\pstate} |\XX_{\TT}(\omega)-\XX_{\TT}(\omega')|\,\diff\Pi(\omega,\omega')\le \int_{\pstate\times\pstate} d(\omega,\omega')\,\diff\Pi(\omega,\omega')=W_{1,d}(Q,\pmeasure),
\end{align*}
which establishes \eqref{eq:infW_W1} by taking the infimum over $Q$.
\end{proof}

\begin{remark}
\label{rem:better_c}
It is possible in principle to improve the constant $\TT\ubd$ to $\LL$ as follows. Instead of working with $\state=[0,\TT]\times [0,\ubd]$, we can work with 
\[
\tilde{\state}:=\left\{(t,z)\in [0,\TT]\times\R_{\ge 0}: z\le \frac{\Pheat_{\TT-t}\den(1)}{\Pheat_{\TT-t}\den(0)}\right\},
\]
since $\ins_t\le \frac{\Pheat_{\TT-t}\den(1)}{\Pheat_{\TT-t}\den(0)}$ $\pmeasure$-a.s.  (cf. \eqref{eq:lambda_M}). Then the volume of  $\tilde{\state}$ is $\int_0^{\TT}\frac{\Pheat_{\TT-t}\den(1)}{\Pheat_{\TT-t}\den(0)}=\LL$, where the equality holds by Corollary \ref{cor:mean_mu}(3). This approach, however, requires a modification of the formulation we used in this paper, with minor benefits, so we do not pursue this improvement. 
\end{remark}

\bibliographystyle{amsplain0}
\bibliography{ref_Poisson_transport}

\end{document}